\documentclass[11pt,bezier]{article}
\usepackage{amsmath,amssymb,amsfonts,euscript,graphicx}

\newtheorem{preproof}{{\bf \indent Proof.}}

\newenvironment{proof}[1]{\begin{preproof}{\rm
               #1}\hfill{$\Box$}}{\end{preproof}}

\newtheorem{prop}{\bf\indent Proposition}[section]
\newtheorem{defn}[prop]{\bf\indent Definition}
\newtheorem{cor}[prop]{\bf\indent Corollary}
\newtheorem{example}[prop]{\bf\indent Example}
\newtheorem{thm}[prop]{\bf\indent Theorem}

\newtheorem{lem}[prop]{\bf\indent Lemma}
\title{\bf  \large On weakly $1$-absorbing prime ideals of commutative rings}
\author{{\normalsize  {\sc M. J. Nikmehr${}^{\mathsf{a}}$, {\sc R. Nikandish${}^{\mathsf{b}}$} and {\sc A. Yassine${}^{\mathsf{a}}$}  }
}\vspace{3mm}\\
{\footnotesize{${}^{\mathsf{a}}$\it Faculty of Mathematics, K.N. Toosi
University of Technology, }}\\
{\footnotesize{\rm P.O. BOX \rm{16315-1618}, Tehran, Iran}}\\
{\footnotesize{ $\mathsf{nikmehr@kntu.ac.ir}$}}\quad\quad
{\footnotesize{$\mathsf{yassine\_ali@email.kntu.ac.ir}$}}\\
{\footnotesize{${}^{\mathsf{b}}$\it Department of Mathematics, Jundi-Shapur University of Technology,}}\\
{\footnotesize{\rm P.O. BOX \rm{64615-334},
Dezful, Iran}}\\
{\footnotesize{ $\mathsf{r.nikandish@ipm.ir}$}}\\
{\footnotesize{$\mathsf{}$ }}}
\date{}

\begin{document}

\maketitle
\begin{abstract}
{
Let $R$ be a commutative ring with identity. In this paper, we introduce the concept of weakly $1$-absorbing prime ideals which is a generalization of weakly prime ideals. A proper ideal $I$ of $R$ is called weakly $1$-absorbing prime if for all nonunit elements $a,b,c \in R$ such that $0\neq abc \in I$, then either $ab \in I$ or $c \in I$. A number of results concerning weakly $1$-absorbing prime ideals and examples of weakly $1$-absorbing prime ideals are given.  It is proved that if $I$ is a weakly $1$-absorbing prime ideal of a ring $R$ and $0 \neq I_1I_2I_3 \subseteq I$ for some ideals $I_1, I_2, I_3$ of $R$ such that $I$ is free triple-zero with respect to $I_1I_2I_3$, then $ I_1I_2 \subseteq I$ or $I_3\subseteq I$. Among other things, it is shown that if $I$ is a weakly $1$-absorbing prime ideal of $R$ that is not $1$-absorbing prime, then $I^3 = 0$. Moreover, weakly $1$-absorbing prime ideals of PID's and Dedekind domains are characterized. Finally, we investigate commutative rings with the property that all proper ideals are weakly $1$-absorbing primes.}
\end{abstract}
\begin{center}{\section{Introduction
}}\end{center}
\par
We assume throughout this paper that all rings are commutative with identity. Let $R$ be a ring and $I$ be an ideal of $R$. The set of nilpotent elements of $R$, the set of zero-divisors of $R$, the set of integers, and integers modulo $n$ are denoted by $\sqrt{0}$, $Z(R)$, $\mathbb{Z}$ and $\mathbb{Z}_n$, respectively. By a proper ideal $I$ of $R$ we mean an ideal with $I\neq R$. A ring $R$ is called \textit{local} if it has a unique maximal ideal. A ring $R$ is called a \textit{reduced} ring if it has no non-zero nilpotent elements; i.e., $\sqrt{0} = 0$. For any undefined notation or terminology in commutative ring theory, we refer the reader to \cite{sha}.

In recent years, various generalizations of prime ideals have been studied by several authors. For instance, in 1978, Hedstrom and Houston \cite{Hedstrom} defined the strongly prime ideal, that is a proper ideal $P$ of $R$ such that for $a,b \in K$ with $ab \in P$, either $a \in P$ or $b \in P$ where $K$ is the quotient field of $R$. In 2003, Anderson and Smith \cite{Ahmet} introduced the notion of a weakly prime ideal, i.e., a proper ideal $P$ of $R$ with the property that for $a, b \in R$, $0 \neq ab \in P$ implies $a \in P$ or $b \in P$. So a prime ideal is weakly prime. In 2005, Bhatwadekar and Sharma \cite{Bhatwadekar} introduced the notion of almost prime ideal which is also a generalization of prime ideal. A proper ideal $I$ of an integral domain $R$ is said to be almost prime if for $a,b \in R$ with $ab \in I \setminus I^2$, then either $a \in I$ or $b \in I$, and it is clear that every weakly prime ideal is an almost prime ideal. Another generalization of  prime ideal is $2$-prime; Indeed, a nonzero proper ideal $I$ of $R$ is called  $2$-prime if whenever $a,b \in R$ and $ab \in I$, then $a^2 \in I$ or $b^2 \in I$ (See \cite{Beddani} and \cite{Nikandish} for more details). The notion of $2$-absorbing ideals was introduced and investigated in 2007 by Badawi \cite{Badawi}. A nonzero proper ideal $I$ of $R$ is called $2$-absorbing if whenever $a,b,c \in R$ and $abc \in I$, then $ab \in I$ or $ac \in I$ or $bc \in I$. In \cite{Darani}, Badawi and Darani extended the concept of weakly prime ideal to  weakly $2$-absorbing ideal. A proper ideal $I$ of $R$ is said to be a weakly $2$-absorbing ideal of $R$ if whenever $a,b,c \in R$ with $0 \neq abc \in I$ implies $ab \in I$ or $ac \in I$ or $bc \in I$. In \cite{Bad}, $1$-absorbing primary ideal was introduced and studied. A proper ideal $I$ of $R$ is called $1$-absorbing primary if for all nonunit $a,b,c \in R$ such that $abc \in I$, then either $ab \in I$ or $c \in \sqrt{I}$. Recall that $1$-absorbing prime ideals, as a generalization of prime ideals, were introduced and investigated in \cite{yassine}. In this paper, we extend the concepts of weakly prime ideal and $1$-absorbing primary ideal to  weakly $1$-absorbing prime ideal.

A proper ideal $I$ of $R$ is called \textit{(weakly) $1$-absorbing prime} if for all nonunit elements $a,b,c \in R$ such that $(0\neq) abc \in I$, then either $ab \in I$ or $c \in I$. Clearly, every weakly prime ideal is a weakly $1$-absorbing prime ideal. However, the converse is not true.

This paper is organized as follows. In section $2$, first we give the definition of weakly $1$-absorbing prime ideals (Definition \ref{dfn}). For nontrivial weakly $1$-absorbing prime ideals see Example \ref{ex}. Also, it is proved (Theorem \ref{3.5}) that if $I$ is a weakly $1$-absorbing prime ideal of $R$ that is not $1$-absorbing prime, then $I^3 = 0$. In Theorem \ref{3888} we give a condition under which every weakly $1$-absorbing prime ideal of $R$ is $1$-absorbing prime. Among other things, it is shown (Theorem \ref{min}) that the radical of a weakly $1$-absorbing prime ideal of a ring $R$ need not be a prime ideal of $R$. Finally, it is proved (Theorem \ref{fin}) that if $I$ is a weakly $1$-absorbing prime ideal of a ring $R$ and $0 \neq I_1I_2I_3 \subseteq I$ for some ideals $I_1, I_2, I_3$ of $R$ such that $I$ is free triple-zero with respect to $I_1I_2I_3$, then $ I_1I_2 \subseteq I$ or $I_3\subseteq I$. In Section $3$, we characterize rings with the property that all proper ideals are weakly $1$-absorbing prime.

\vspace*{1cm}

\begin{center}{\section{Weakly $1$-absorbing prime ideals
}}\end{center}

 In this section, the concept of weakly $1$-absorbing prime ideals is introduced and investigated. We start with the following definitions.

\begin{defn} \label{dfn}
Let $R$ be a ring. A proper ideal $I$ of $R$ is called weakly $1$-absorbing prime if for all nonunit elements $a,b,c \in R$ such that $0 \neq abc \in I$, then $ab \in I$ or $c \in I$.
\end{defn}
\begin{defn}
 Suppose that $I$ is a weakly $1$-absorbing prime ideal of a ring $R$ and $a,b,c \in R$ are nonunit elements.

$(1)$ We say $(a,b,c)$ is a triple-zero of $I$ if $abc = 0$, $ab \notin I$ and $c \notin I$.

$(2)$  Suppose that $I_1I_2I_3 \subseteq I$, for some ideals $I_1, I_2, I_3$ of $R$. We say $I$ is free triple-zero with respect to $I_1I_2I_3$, if $(a,b,c)$ is not a triple-zero of $I$, for every $a \in I_1, b \in I_2$, and $c \in I_3$.
\end{defn}

It is easy to see that if $I$ is a weakly $1$-absorbing prime ideal that is not $1$-absorbing prime, then there exists a triple-zero of $I$. Note that every weakly prime ideal is a weakly $1$-absorbing prime ideal, and every weakly $1$-absorbing prime ideal is a weakly $2$-absorbing prime ideal. The following  example provides  a weakly $1$-absorbing prime ideal that is not $1$-absorbing prime,  a weakly $2$-absorbing ideal that is not weakly $1$-absorbing prime and a weakly $1$-absorbing prime ideal that is not weakly prime.

\begin{example}\label{ex} {\rm
$(1)$ Consider the ideal $J = \{0, 4\}$  of $\mathbb{Z}_{8}$. Let  $R = \mathbb{Z}_8 (+) J$ and $I = \{(0, 0), (0, 4)\}$. It is easy to see that $I$ is a weakly $1$-absorbing prime ideal, since $abc \in I$ for some $a,b,c \in R \setminus I$ if and only
if $abc = (0, 0)$. But, $(2, 0)(2, 0)(2, 0) \in I$ and both $(4, 0) \notin I$ and $(2, 0) \notin I$, thus $I$ is a weakly $1$-absorbing prime ideal that is not $1$-absorbing prime.

 $(2)$ Suppose that $R = R_1 \times R_2 \times R_3$, where $R_1, R_2, R_3$ are fields and $I = R_1 \times \{0\} \times \{0\}$. One can easily see that the ideal $I$ is not weakly $1$-absorbing prime, since $(0, 0, 0)\neq (1, 0, 1)(1, 0, 1)(1, 1, 0) = (1, 0, 0)\in I$ and neither $(1, 0, 1)(1, 0, 1) = (1, 0, 1)\in I$ nor $(1, 1, 0)\in I$. But $I$ is a weakly $2$-absorbing ideal of $R$, by \cite[Theorem 3.5]{Darani}. Hence every weakly $2$-absorbing ideal need not be weakly $1$-absorbing prime.

$(3)$ Consider the ideal $J = \{0, 4, 8\}$ of the ring $\mathbb{Z}_{12}$. Clearly,
$J$ is a weakly $1$-absorbing prime ideal of $\mathbb{Z}_{12}$ that is not weakly prime.
 }
\end{example}
\begin{prop}
Let $R$ be a ring, $I$ a weakly $1$-absorbing prime ideal of $R$ and $c$ be a nonunit element of $R\setminus I$. Then $(I : c)$ is a weakly prime ideal of $R$.
\end{prop}
\begin{proof}
{ Assume that $0 \neq ab \in (I : c)$ for some nonunit element $c \in R\setminus I$ such that $a\notin (I : c)$. We may assume that $a,b$ are nonunit elements of $R$. As $0 \neq abc = acb \in I$ and $ac \notin I$ and $I$ is a weakly $1$-absorbing prime ideal of $R$, we have $b \in I \subseteq (I : c)$. Hence $(I : c)$ is a weakly prime ideal of $R$.
}
\end{proof}
 Let $R$ be a local ring and $I$ be a weakly $1$-absorbing prime ideal of $R$ which is not $1$-absorbing prime. To prove  $I^3 = 0$, the following lemma is needed.

\begin{lem} \label{3.4}
 Let $I$ be a weakly $1$-absorbing prime ideal of a local ring $R$ and let $(a,b,c)$ be a triple-zero of $I$ for some nonunit elements $a,b,c \in R$. Then the following statements hold.

$(1)$ $abI = acI = bcI = 0$.

$(2)$ $aI^2 = bI^2 = cI^2 = 0$.
\end{lem}
\begin{proof}
{ $(1)$ Assume that $abI \neq 0$. Then there exists $x \in I$ such that $abx \neq 0$. Therefore $ab(c+x) \neq 0$. Since $ab \notin I$ and $I$ is a weakly $1$-absorbing prime ideal, $(c + x) \in I$, and hence $c \in I$, a contradiction. Thus $abI = 0$. Now suppose that $acI \neq 0$. Then there exists $x \in I$ such that $acx \neq 0$. Therefore $a(b+x)c \neq 0$. Since $c \notin I$ and $I$ is a weakly $1$-absorbing prime ideal, $a(b + x) \in I$, and so $ab \in I$, a contradiction. Thus $acI = 0$. Similarly, one can easily see that $bcI = 0$.

$(2)$ Assume that $axy \neq 0$ for some $x,y \in I$. By part $(1)$, $a(b + x)(c + y) = axy \neq 0$. Since $R$ is local, the set of nonunit elements of $R$ is an ideal of $R$. Therefore $(b + x), (c + y)$ are nonunit elements of $R$. This means that either $a(b + x) \in I$ or $c + y \in I$. Hence either $ab \in I$ or $c \in I$, a contradiction.
Thus $aI^2 = 0$. Similarly, one can easily show that  $bI^2 = cI^2 = 0$.
}
\end{proof}

\begin{thm} \label{3.5}
 Let $I$ be a weakly $1$-absorbing prime ideal of a local ring $R$ that is not $1$-absorbing prime. Then $I^3 = 0$.
\end{thm}
\begin{proof}
{  Suppose that $I$ is a weakly $1$-absorbing prime ideal of a ring $R$ that is not $1$-absorbing prime. Then there exists a triple-zero $(a,b,c)$ of $I$ for some nonunit elements $a,b,c \in R$. Suppose that $I^3 \neq 0$. Hence $xyz \neq 0$ for some $x,y,z \in I$. By Lemma \ref{3.4}, $(a + x)(b + y)(c + z) = xyz \neq 0$. Since $R$ is local, the set of nonunit elements of $R$ is an ideal of $R$. Therefore $(a+x), (b+y)$ and $(c+z)$ are nonunit elements. Hence either $(a+ x)(b+ y) \in I$ or $(c + z) \in I$, and so either $ab \in I$ or $c \in I$, a contradiction. Thus $I^3 = 0$.
}
\end{proof}

It is worth mentioning that if $I$ is an ideal of a ring $R$ such that $I^3 = 0$, then $I$ need not be a weakly $1$-absorbing prime ideal. For example, let $R = \mathbb{Z}/16\mathbb{Z}$ and $I = 8\mathbb{Z}/16\mathbb{Z}$. Then $I^3 = 0$, but
$\overline{0}\neq \overline{2}\cdot \overline{2}\cdot \overline{2} = \overline{8} \in I$ and $\overline{2},\overline{4}\notin I$.

Now, we state the following corollary.
\begin{cor} \label{36}
$(1)$ Let $I$ be a weakly $1$-absorbing prime ideal of a local ring $R$ that is not $1$-absorbing prime. Then $\sqrt{I} = \sqrt{0}$.

$(2)$  If $R$ is a reduced local ring and $I \neq 0$ is a proper ideal of $R$, then $I$ is a weakly $1$-absorbing prime ideal of $R$ if and only if $I$ is a $1$-absorbing prime ideal of $R$.

\end{cor}

\begin{thm} \label{nwe}
 Let $I$ be a weakly $1$-absorbing prime ideal of a local ring $R$ that is not $1$-absorbing prime. Then the following statements hold.

$(1)$ If $w \in \sqrt{0}$, then either $w^2 \in I$ or $w^2I = wI^2 = \{0\}$.

$(2)$ $\sqrt{0}^2I^2 = \{0\}$.
\end{thm}
\begin{proof}
{$(1)$ Suppose that $w \in \sqrt{0}$ such that $w^2I \neq \{0\}$. Suppose also that $n$ is the least positive integer such that
$w^n = 0$. Then $n \geq 3$ and $w^2(x + w^{n-2}) = w^2x \neq 0$ for some $x \in I$. Therefore,
either $w^2 \in I$ or $(x + w^{n-2})\in I$, since $R$ is local (which means that $(x + w^{n-2})$ is nonunit) and $I$ is weakly $1$-absorbing prime. If $(x + w^{n-2})\in I$, then $w^{n-2}\in I$, and so $w^{2}\in I$, as $w^{n-4}\notin I$. Hence $w^2 \in I$. Thus either $w^2 \in I$ or $w^2I = \{0\}$, for every $w \in \sqrt{0}$. Next, we show that $wI^2 = \{0\}$. Suppose that $wI^2 \neq \{0\}$ and $w^2 \notin I$. Then $w^2I = \{0\}$ and $wxy \neq 0$ for some $x,y\in I$. Let $m$ be the least positive integer such that $w^m = 0$. This means that $m \geq 3$ and $w^2I = 0$. Therefore $w(w + x)(w^{m-2} + y) = wxy \neq 0$. But $R$ is local (which means that $(y + w^{n-2})$ is non unit) and $I$ is weakly $1$-absorbing prime, hence either $w^2 \in I$ or $w^{m-2}\in I$,  a contradiction. Thus $wI^2 = \{0\}$.

$(2)$ Suppose that $a,b \in \sqrt{0}$. It follows from part $(1)$ that if either $a^2 \notin I$ or $b^2 \notin I$, then $abI^2 = \{0\}$. Hence assume that $a^2 \in I$ and $b^2 \in I$, which means that $ab(a + b) \in I$. We distinguish two cases, according as $(a,b,a + b)$ is
a triple-zero of $I$ or not. If $(a,b,a + b)$ is a triple-zero of $I$, then $abI = \{0\}$, by Lemma \ref{3.4} (1), and thus $\sqrt{0}^2I^2 = \{0\}$. If $(a,b,a+b)$ is not a triple-zero of $I$, then it is easily seen that $ab \in I$, and so $\sqrt{0}^2I^2 = \{0\}$, by Theorem \ref{3.5}.
}
\end{proof}
Corollary \ref{36} and  Theorem \ref{nwe} lead to the following corollary.
\begin{cor}
Let $I,J,K$ be weakly $1$-absorbing prime ideals of a local ring $R$ such that none of them is $1$-absorbing prime.  Then $I^2JK = IJ^2K = IJK^2 = I^2J^2 = I^2K^2 = J^2K^2 = \{0\}$.
\end{cor}

 If $I$ is a proper ideal of $R$, it is well known that $\sqrt{I}$ is a prime ideal of $R$ if and only if $\sqrt{I}$ is a primary ideal of $R$. In the following result, we give a condition under which every weakly $1$-absorbing prime ideal of $R$ is $1$-absorbing prime.

\begin{thm} \label{3888}
Suppose that $I$ is a proper ideal of $R$ such that $\sqrt{0} \subseteq I$ and $\sqrt{0}$ is a prime (primary) ideal of $R$. Then $I$ is a weakly $1$-absorbing prime ideal of $R$ if and only if $I$ is a $1$-absorbing prime ideal of $R$.
\end{thm}
\begin{proof}
{  Suppose that $I$ is a weakly $1$-absorbing prime ideal of the  ring $R$ and $xyz \in I$ for some nonunit elements $x,y,z \in R$. If $xyz \neq 0$, then either $xy \in I$ or $z \in I$. Therefore assume that $xyz = 0$ and $z \notin I$. Since $xyz = 0 \in \sqrt{0}$ which is a prime ideal of $R$, we conclude that $xy \in \sqrt{0} \subseteq I$. Thus $I$ is a weakly $1$-absorbing prime ideal of $R$ if and only if $I$ is a $1$-absorbing prime ideal of $R$.}
\end{proof}

In the following result, we give a condition under which a weakly $1$-absorbing prime ideal of $R$ is not $1$-absorbing prime.

\begin{thm}
Suppose that $I$ is a weakly $1$-absorbing prime ideal of a local ring $R$ and $\{0\}$ has a triple-zero $(a,b,c)$ for some nonunit elements $a,b,c \in R$ such that $ab \notin \sqrt{0}$. Then $I$ is not a $1$-absorbing prime ideal of $R$ if and only if $I \subseteq \sqrt{0}$.
\end{thm}
\begin{proof}
{  Suppose that $I$ is not a $1$-absorbing prime ideal of $R$. Hence, by part $(1)$ of Corollary \ref{36}, $I \subseteq \sqrt{0}$. Conversely, suppose that $I \subseteq \sqrt{0}$ and $\{0\}$ has a triple-zero $(a,b,c)$ for some nonunit elements $a,b,c \in R$ such that $ab \notin \sqrt{0}$. Then $ab \notin I$ and $c \notin I$. Hence $(a,b,c)$ is a triple-zero of $I$, and thus $I$ is not a $1$-absorbing prime ideal of $R$.}
\end{proof}

One can easily see that if $I$ is a $1$-absorbing prime ideal of a ring $R$, then there exists just one prime ideal of
$R$ that is minimal over $I$. In the following result, we show that for every positive integer $n \geq 2$, there exist a ring $R$ and a nonzero weakly $1$-absorbing prime ideal $I$ of $R$ such that there are exactly $n$ prime ideals of $R$
that are minimal over $I$.

\begin{thm} \label{min}
Suppose that $n \geq 2$ is a positive integer. Then there exist a ring $R$ and a nonzero weakly $1$-absorbing prime ideal $I$ of $R$ such that $I$ has exactly $n$ minimal prime ideals of $R$.
\end{thm}
\begin{proof}
{ Suppose that $n \geq 2$ is a positive integer and $D = \mathbb{Z}_8 \times \cdots \times \mathbb{Z}_8$ ($n$ times). Clearly, $M = \{\overline{0},\overline{4}\}$ is an ideal of $\mathbb{Z}_8$. Define the $D$-module $M$ such that $xM = a_1M$ for every $x = (a_1,\ldots , a_n) \in D$ and consider the idealization ring $R = D (+) M$ and the ideal $I = \{(0,\ldots ,0)\}(+)M$ of $R$. It follows from \cite[Theorem 25.1 (3)]{Huckaba} that every prime ideal of $R$ is of the form $P(+)M$ for some prime ideal $P$ of $D$ and since for every $a,b,c \in R\setminus I$ and $abc \in I$, we deduce that $abc = ((0,\ldots ,0),0)$. Hence $I$ is a nonzero weakly $1$-absorbing prime ideal of $R$, and thus there are exactly $n$ prime ideals of $R$ that are minimal over $I$.
}
\end{proof}

In the next result, one can see that in a reduced ring, the radical of a nonzero weakly $1$-absorbing prime ideal a prime ideal.

\begin{thm} \label{reduced}
 Let $R$ be a reduced ring and $I$ a nonzero weakly $1$-absorbing prime ideal of $R$. Then $\sqrt{I}$ is a prime ideal of $R$.
\end{thm}
\begin{proof}
{ Let $0 \neq ab \in \sqrt{I}$ for some $a,b \in R$. Without loss of generality, we may assume that $a, b$ are nonunit. Then $(ab)^n \in I$ for some positive integer $n$, and so $(ab)^n \neq 0$, because $R$ is reduced. Therefore $0 \neq a^ma^{(n-m)}b^n\in I$ for some positive integer $m$, and hence either $a^ma^{(n-m)} = a^n \in I$ or $b^n \in I$. Thus $\sqrt{I}$ is a weakly prime ideal of $R.$ But $R$ is reduced and $I \neq \{0\}$, hence $\sqrt{I}$ is a prime ideal of $R$, by \cite[Corollary 2]{Ahmet}.
}
\end{proof}

 It was shown in \cite[Theorem 2.4]{yassine} that if $R$ is a non-local ring, then every $1$-absorbing prime ideal is prime. In the following theorem, we show that if $R$ is a non-local ring and $I$ is a proper ideal of $R$ having the property that $ann(x)$ is not a maximal ideal of $R$ for every element $x \in I$, then $I$ is a weakly $1$-absorbing prime ideal if and only if $I$ is a weakly prime ideal.

\begin{thm} \label{reduced2}
Let $R$ be a non-local ring and $I$ a proper ideal of $R$ having the property that $ann(x)$ is not a maximal ideal of $R$ for every element $x \in I$. Then $I$ is a weakly $1$-absorbing prime ideal if and only if $I$ is a weakly prime ideal.
\end{thm}
\begin{proof}
{  It is easy to see that every weakly prime ideal of $R$ is a weakly $1$-absorbing prime ideal of $R$. Then assume that $I$ is a weakly $1$-absorbing prime ideal of $R$ and let $0 \neq ab \in I$ for some $a,b \in R$. Without loss of generality, we may assume that $a, b$ are nonunit. Since $ab \neq 0$, $ann(ab)$ is a proper ideal of $R$, and so $ann(ab) \subset L$, for some  maximal ideal $L$ of $R$. But $R$ is a non-local ring, hence there exists a maximal ideal $M$ of $R$ such that $M \neq L$. Suppose that $m \in M \setminus L$. Then $m \notin ann(ab)$, which means that $0 \neq mab \in I$. Since $I$ is a weakly $1$-absorbing prime ideal of $R$, either $ma \in I$ or $b\in I$. If $b\in I$, then the proof is complete. So suppose that $b \notin I$. Therefore $ma\in I$. But $m \notin L$ and $L$ is a maximal ideal of $R$, hence $m \notin$ J$(R)$. This means that there exists an $r \in R$ such that $1 + rm$ is a nonunit element of $R$. We take the following two cases: \textbf{Case one:} if $1+rm \notin ann(ab)$, then $0 \neq (1+rm)ab \notin I$. Since $I$ is a weakly $1$-absorbing prime ideal of $R$ and $b \notin I$, we conclude that $(1 + rm)a = a + rma \in I$. But $rma \in I$, so $a \in I$ and $I$ is a weakly prime ideal of $R$. \textbf{Case two:} Suppose that $1+rm \in ann(ab)$. Since $ann(ab)$ is not a maximal ideal of $R$ and $ann(ab) \subset L$, there exists an element $w \in L \setminus ann(ab)$. Therefore $0 \neq wab  \in I$. But $I$ is a weakly $1$-absorbing prime ideal of $R$ and $b \notin I$, thus $wa \in I$. We have $1+rm+w$ is a nonzero nonunit element of $L$, because $1 + rm \in ann(ab) \subset L$ and $w \in L \setminus ann(ab)$. Thus $0 \neq (1+rm+w)ab \in I$. Now, since $I$ is a weakly $1$-absorbing prime ideal of $R$ and $b \notin I$, we have $(1+rm+w)a = a+rma+wa \in I$. Hence $a \in I$, and thus $I$ is a weakly prime ideal of $R$.
}
\end{proof}

Let $R$ be an integral domain with the quotient field $K$. Recall that a proper ideal $I$
of $R$ is called invertible if $II^{-1} = R$, where $I^{-1} = \{r \in K : rI \subseteq R\}$. An integral domain is called a Dedekind domain if every nonzero proper ideal of $R$ is invertible. In the following results, weakly $1$-absorbing prime ideals of Dedekind domains and principal ideal domains are completely described.

\begin{thm} \label{resulta}
Let $R$ be a Noetherian integral domain that is not a field and $I$ an ideal of $R$. Then $(1) \Rightarrow (2) \Rightarrow (3)$.

$(1)$ $R$ is a Dedekind domain;

$(2)$ If $I$ is a weakly $1$-absorbing prime ideal of $R$, then $I = M$ or $I = M^2$ where $M$ is a maximal ideal of $R$;

$(3)$ If $I$ is a weakly $1$-absorbing prime ideal of $R$, then $I = P$ or $I = P^2$ where $P = \sqrt{I}$ is a prime ideal of $R$.
\end{thm}
\begin{proof}
{$(1) \Rightarrow (2)$  Suppose that $R$ is a Noetherian integral domain that is not a field and $I$ is a weakly $1$-absorbing prime ideal of $R$ such that $P = \sqrt{I}$. Since $R$ is a Dedekind domain, we conclude that every nonzero prime ideal of $R$ is a maximal ideal of $R$. Therefore $P$ is a maximal ideal of $R$, by Theorem \ref{reduced}. This means that $I$ is a primary ideal of $R$ such that $P^2 \subseteq I$. Hence, by \cite[Theorem 2.10]{yassine}, $I = M$ or $I = M^2$ where $M$ is a maximal ideal of $R$.

$(2) \Rightarrow (3)$ This is obvious.}
\end{proof}

In view of Theorem \ref{resulta}, we have the following result.

\begin{cor}
Let $R$ be a principal ideal domain that is not a field and $I$ be a nonzero proper ideal of $R$. If $I$ is a weakly $1$-absorbing prime ideal of $R$, then $I = pR$ or $I = p^2R$ for some nonzero prime element $p$ of $R$.
\end{cor}

In the following theorems we show that weakly $1$-absorbing prime
ideals are really of interest in indecomposable rings.

\begin{thm} \label{311}
  Suppose that $R = R_1 \times R_2$ is a decomposable ring, where $R_1$ and $R_2$ are rings  and $I$ is a proper ideal of $R_1$. Then the following statements are equivalent.

$(1)$ $I \times R_2$ is a weakly $1$-absorbing prime ideal of $R$.

$(2)$ $I \times R_2$ is a $1$-absorbing prime ideal of $R$.

$(3)$ $I$ is a $1$-absorbing prime ideal of $R_1$.
\end{thm}
\begin{proof}
{ $(3) \Rightarrow (2)$ and $(2) \Rightarrow (1)$ These are clear.

$(1) \Rightarrow (3)$ Suppose that $I \times R_2$ is a weakly $1$-absorbing prime ideal of $R$ and $abc \in I$ for some nonunit elements $a,b,c \in R_1$. Suppose also that $0\neq x\in R_2$. Then $(0,0)\neq (a,1)(b,1)(c,x) = (abc,x) \in I \times R_2$, and so either $(a,1)(b,1) = (ab,1) \in I \times R_2$ or $(c,x) \in I\times R_2$. Hence, either $ab \in I$ or $c \in I$. Thus $I$ is a $1$-absorbing prime ideal of $R_1$.
}
\end{proof}

\begin{thm}  \label{312}
  Suppose that $R = R_1 \times R_2$ is a decomposable ring, where $R_1$ and $R_2$ are rings  and $I_1, I_2$ are nonzero ideals of $R_1$ and $R_2$, respectively. Then the following statements are equivalent.

$(1)$ $I_1 \times I_2$ is a weakly $1$-absorbing prime ideal of $R$.

$(2)$  $I_1 = R_1$ and $I_2$ is a $1$-absorbing prime ideal of $R_1$ or $I_2 = R_2$ and $I_1$ is a $1$-absorbing prime ideal of $R_1$ or $I_1, I_2$ are prime ideals of $R_1, R_2$, respectively.

$(3)$ $I_1 \times I_2$ is a $1$-absorbing prime ideal of $R$.

$(4)$ $I_1 \times I_2$ is a prime ideal of $R$.
\end{thm}
\begin{proof}
{ $(1) \Rightarrow (2)$ Let $I_1 \times I_2$ be a weakly $1$-absorbing prime ideal of $R$. If $I_1 = R_1$ ($I_2 = R_2$), then by Theorem \ref{311}, $I_2$ ($I_1$) is a $1$-absorbing prime ideal of $R_2$ ($R_1$). Therefore assume that $I_1$ and $I_2$ are proper ideals. Suppose that $a,b \in R_2$ such that $ab \in I_2$ and $0 \neq x \in I_1$. Then $(0,0) \neq (x,1)(1,a)(1,b) = (x,ab) \in I_1 \times I_2$. But $I_1$ is proper, hence $(1,a) \notin I_1\times I_2$ and $(1,b) \notin I_1\times I_2$. Without loss of generality, we may assume that $x, a, b$ are nonunit. Since $I_1 \times I_2$ is a weakly $1$-absorbing prime ideal of $R$, we conclude that $(x,1)(1,a) = (x,a) \in I_1 \times I_2$. Thus $I_2$ is a prime ideal of $R_2$. Similarly, it can be easily shown that $I_1$ is a prime ideal of $R_1$.

$(2) \Rightarrow (3)$ If $I_1 = R_1$ and $I_2$ is a $1$-absorbing prime ideal of $R_2$ or $I_2 = R_2$ and $I_1$ is a $1$-absorbing prime ideal of $R_1$, then the result follows from Theorem \ref{311}. Assume that $I_1, I_2$ are prime ideals of $R_1, R_2$, respectively. Since $I_1, I_2$ are nonzero ideals of $R_1$ and $R_2$, assume that $(0, 0)\neq (a, b) \in I_1 \times I_2$. Then $(0,0) \neq (a, 1)(a, 1)(1, b) = (a^2, b) \in I_1 \times I_2$. This means that either $(a, 1)(a, 1) \in I_1 \times I_2$ or $(1, b) \in I_1 \times I_2$, since $I_1 \times I_2$ is a weakly $1$-absorbing prime ideal of $R$. Hence either $(a, 1) \in I_1 \times I_2$ or $(1, b) \in I_1 \times I_2$, as $I_1$ is a prime ideals of $R_1$, a contradiction. Thus $I_1 \times I_2 = R_1 \times I_2$ or $I_1 \times I_2 = I_1 \times R_2$ and in both cases $I_1 \times I_2$ is a $1$-absorbing prime ideal of $R$.

$(3) \Rightarrow (4)$ Since $R$ is not a local ring, the proof follows from \cite[Theorem 2.4]{yassine}.

$(4) \Rightarrow (1)$ This is straightforward.}
\end{proof}

\begin{thm} \label{2131}
  Suppose that $R = R_1 \times R_2$ is a decomposable ring, where $R_1$ and $R_2$ are rings and $I_1$ is a nonzero proper ideal of $R_1$ and $I_2$ is a proper ideal of $R_2$. Then $(1) \Rightarrow (2)$. In fact, $(2) \Rightarrow (1)$ need not be true.

$(1)$ $I_1 \times I_2$ is a weakly $1$-absorbing prime ideal of $R$ that is not a $1$-absorbing prime ideal.

$(2)$ $I_1$ is a weakly prime ideal of $R_1$ that is not a prime ideal and $I_2 = \{0\}$ is
a prime ideal of $R_2$.
\end{thm}
\begin{proof}
{ $(1) \Rightarrow (2)$ Suppose that $I_1 \times I_2$ is a weakly $1$-absorbing prime ideal of $R$ that
is not a $1$-absorbing prime ideal and $I_2 \neq \{0\}$. Hence, by Theorem \ref{312}, $I_1 \times I_2$ is a $1$-absorbing
prime ideal of $R$, a contradiction, and thus $I_2 = \{0\}$. Now, we show that $I_2$ is a prime ideal of $R_2$. Suppose that $a,b \in R_2$ such that $ab \in I_2$ and $0 \neq x \in I_1$. Then $(0,0) \neq (x,1)(1,a)(1,b) = (x,ab) \in I_1 \times I_2$. But $I_1$ is proper, hence $(1,a) \notin I_1\times I_2$ and $(1,b) \notin I_1\times I_2$. Without loss of generality, we may assume that $x, a, b$ are nonunit. Since $I_1 \times I_2$ is a weakly $1$-absorbing prime ideal of $R$,  $(x,1)(1,a) = (x,a) \in I_1 \times I_2$. Thus $I_2 = \{0\}$ is a prime ideal of $R_2$. We show that $I_1$ is a weakly prime ideal of $R_1$. Let $a,b\in R_1$ such that $0\neq ab \in I_1$. Without loss of generality, we may assume that $a, b$ are nonunit. Since $(0,0) \neq (b,1)(1,0)(ab,1) = (ab^2,0) \in I_1 \times I_2$ and $(ab,1) \notin I_1\times \{0\}$ and $I_1 \times I_2$ is a weakly $1$-absorbing prime ideal of $R$, we conclude that $(b,0) \in I_1 \times \{0\}$, and hence $b\in I_1$. Thus $I_1$ is a weakly prime ideal of $R_1$. Assume that $I_1$ is a prime ideal of $R_1$. Since $I_1$ is a nonzero ideal of $R_1$, assume that $0\neq a \in I_1$. Then $(0,0) \neq (a, 1)(a, 1)(1, 0) = (a^2, 0) \in I_1 \times I_2$. This means that either $(a, 1)(a, 1) \in I_1 \times I_2$ or $(1, 0) \in I_1 \times I_2$, as $I_1 \times I_2$ is a weakly $1$-absorbing prime ideal of $R$. Hence either $(a, 1) \in I_1 \times I_2$ or $(1, 0) \in I_1 \times I_2$, as $I_1$ is a prime ideal of $R_1$, a contradiction. Thus $I_1$ is a weakly prime ideal of $R_1$ that is not a prime ideal and $I_2 = \{0\}$ is a prime ideal of $R_2$.

$(2) \Rightarrow (1)$ Assume that $I_1$ is a weakly prime ideal of $R_1$ that is not a prime ideal and $I_2 = \{0\}$ is
a prime ideal of $R_2$. We show that $I_1 \times I_2$ need not be a weakly $1$-absorbing prime ideal of $R$. Since $I_1 \neq 0$, there exists $0\neq x\in I_1$, and so $(0, 0)\neq (a,1)(a,1)(1,0) = (a^2, 0) \in I_1 \times \{0\}$. Since neither $(a,1)(a,1) = (a^2, 1) \in I_1 \times \{0\}$ nor $(1,0) \notin I_1 \times \{0\}$ and $I_1$ is proper, we conclude that $I_1 \times I_2$ is not a weakly $1$-absorbing prime ideal of $R$.
}
\end{proof}


In the following, we show that in a decomposable ring $R = R_1 \times R_2 \times R_3$ every nonzero weakly $1$-absorbing prime ideal of $R$ is $1$-absorbing prime.

\begin{thm} \label{315}
  Suppose that $R = R_1 \times R_2 \times R_3$ is a decomposable ring, where $R_1$, $R_2$ and $R_3$ are rings and $I = I_1 \times I_2 \times I_3$ is a nonzero proper ideal of $R$. Then $I$ is a weakly $1$-absorbing prime ideal if and only if $I$ is a $1$-absorbing prime ideal.

\end{thm}
\begin{proof}
{Let $I = I_1 \times I_2 \times I_3$ be a nonzero weakly $1$-absorbing prime ideal of $R$. Then there exists an element
$(0,0,0) \neq (a,b,c) \in I$. Since $(a,1,1)(1,b,1)(1,1,c) = (a,b,c)$,  either
$(a,b,1) \in I$ or $(1,1,c) \in I$. Hence either $I_3 = R_3$ or $I_1 = R_1$ and $I_2 = R_2$, and so
$I = I_1 \times I_2 \times R_3$ or $I = R_1 \times R_2 \times I_3$. Thus, by Theorem \ref{311}, $I$ is a $1$-absorbing prime ideal of
$R$. The converse is clear.
}
\end{proof}

\begin{cor}
  Suppose that $R = R_1 \times R_2 \times R_3$ is a decomposable ring, where $R_1$, $R_2$ and $R_3$ are rings and  $I = I_1 \times I_2 \times I_3$ is a nonzero ideal of $R$. Then the following statements are equivalent.

$(1)$ $I = I_1 \times I_2 \times I_3$ is a weakly $1$-absorbing prime ideal of $R$.

$(2)$ $I = I_1 \times I_2 \times I_3$ is a $1$-absorbing prime ideal of $R$.

$(3)$  Either $I = I_1 \times I_2 \times I_3$
such that for some $k \in \{1, 2, 3\}$, $I_k$ is a $1$-absorbing prime ideal of $R_k$, and $I_j = R_j$ for every $j \in \{1, 2, 3\} - \{k\}$, or $I = I_1 \times I_2 \times I_3$ such that for some $k, m \in \{1, 2, 3\}$, $I_k$ is a prime ideal of $R_k$, $I_m$ is a prime ideal of $R_m$, and $I_j = R_j$ for every $j \in \{1, 2, 3\} - \{k,m\}$.

$(4)$ $I = I_1 \times I_2 \times I_3$ is a prime ideal of $R$.
\end{cor}
\begin{proof}
{ $(1) \Leftrightarrow  (2)$ It follows from Theorem \ref{315}.

$(2) \Rightarrow (3)$ Since $I = I_1 \times I_2 \times I_3$ is a $1$-absorbing prime ideal of $R$, we have either $I_3 \neq R_3$ or $I_2 \neq R_2$ or $I_1 \neq R_1$. If $I_3 \neq R_3$, then by the proof of Theorem \ref{315}, $I_2 = R_2$ and $I_1 = R_1$, and so we are done. If $I_2 \neq R_2$, then by the proof of Theorem \ref{315}, either $I_3 = R_3$ and $I_1 = R_1$ or $I_3 = R_3$ and $I_1 \neq R_1$. The first case is clear, so assume that $I_3 = R_3$ and $I_1 \neq R_1$. We show that $I_1$ is a prime ideal of $R_1$ and $I_2$ is a prime of $R_2$. Suppose that $a,b \in R_1$ such that $ab \in I_1$, and $c,d \in R_2$ such that $cd \in I_2$. Then $(0,0,0) \neq (a,1,1)(1,c,1)(b,d,1) = (ab,cd,1) \in I$. Hence either $a\in I_1$ or $b\in I_1$, and thus $I_1$ is a prime ideal of $R_1$. Similarly, since $(0,0,0) \neq (a,1,1)(b,c,1)(1,d,1) = (ab,cd,1) \in I$, we conclude that either $c\in I_2$ or $d\in I_2$. Hence $I_2$ is a prime ideal of $R_2$. Finally, assume that $I_1 \neq R_1$ and $I_3 = R_3$. By an argument similar to that we applied above, we conclude that $I_1$ is a prime ideal of $R_1$ and $I_2$ is a prime ideal of $R_2$.

$(3) \Rightarrow (2)$ This is clear.

$(2) \Leftrightarrow  (4)$ It follows from \cite[Theorem 2.4]{yassine}. }
\end{proof}

 The next theorem states that if $I$ is a weakly $1$-absorbing prime ideal of a ring $R$ and $0 \neq I_1I_2I_3 \subseteq I$ for some ideals $I_1, I_2, I_3$ of $R$ such that $I$ is free triple-zero with respect to $I_1I_2I_3$, then $ I_1I_2 \subseteq I$ or $I_3\subseteq I$. First, we need the following lemma.

\begin{lem} \label{lem}
Let $I$ be a weakly $1$-absorbing prime ideal of a ring $R$. If
$abJ \subseteq I$ for some nonunit elements $a,b\in R$ and a proper ideal $J$ of $R$ such that $(a,b,c)$ is not a triple-zero of $I$ for every $c \in J$, then $ab \in I$ or $J \subseteq I$.
\end{lem}
\begin{proof}
{Suppose that $abJ \subseteq I$, but $ab \notin I$ and $J \nsubseteq I$. Then there exists an element $j \in J\setminus I$. But $(a,b,j)$ is not a triple-zero of $I$ and $abj \in I$ and $ab \notin I$ and $j \notin I$, a contradiction.
}
\end{proof}

\begin{thm} \label{fin}
Suppose that $I$ is a proper ideal of a ring $R$. Then the following statements are
equivalent.

$(1)$ $I$ is a weakly $1$-absorbing prime ideal of $R$.

$(2)$ For any proper ideals $I_1, I_2, I_3$ of $R$ such that $0 \neq I_1I_2I_3 \subseteq I$ and $I$ is free triple-zero with respect to $I_1I_2I_3$, we have either
$I_1I_2 \subseteq I$ or $I_3 \subseteq I$.
\end{thm}
\begin{proof}
{$(1)\Rightarrow (2)$ Suppose that $I$ is a weakly $1$-absorbing prime ideal of $R$ and $0 \neq I_1I_2I_3 \subseteq I$ for some proper ideals $I_1, I_2, I_3$ of $R$ such that $I_1I_2 \nsubseteq I$ and $I$ is free triple-zero with respect to $I_1I_2I_3$. Then there are nonunit elements $a \in I_1$ and $b \in I_2$ such that $ab \notin I$. Since $abI_3 \subseteq I$, $ab \notin I$ and $(a,b,c)$ is not a triple-zero of $I$ for every $c \in J$, it follows from Lemma \ref{lem} that $J \subseteq I$.

$(2) \Rightarrow (1)$ Suppose that $0\neq abc \in I$ for some nonunit elements $a,b,c \in R$ and
$ab \notin I$. Suppose also that $I_1 = aR, I_2 = bR$, and $I_3 = cR$. Then $0\neq I_1I_2I_3 \subseteq I$ and $I_1I_2 \nsubseteq I$. Hence $I_3 = cR\subseteq I$, and thus $c \in I$.
}
\end{proof}
Suppose that $I$ is an ideal of a ring $R$. It was shown in \cite[Theorem 2.7]{yassine}, if $I$ is $1$-absorbing prime and $I_1I_2I_3 \subseteq I$ for some proper ideals $I_1, I_2, I_3$ of $R$, then $I_1I_2 \subseteq I$ or $I_3 \subseteq I$. We end this section with the following question: if $I$ is weakly $1$-absorbing prime and $0\neq I_1I_2I_3 \subseteq I$ for some proper ideals $I_1, I_2, I_3$ of $R$, does it imply that $I_1I_2 \subseteq I$ or $I_3 \subseteq I$?

\begin{center}{\section{Rings in which every proper ideal is weakly $1$-absorbing prime
}}\end{center}

In this section we study rings in which every proper ideal is weakly $1$-absorbing prime. To prove Theorem \ref{442}, the following lemma is needed.

\begin{lem} \label{44}
 Let $R$ be a ring. Then for every $a,b,c \in J(R)$, the ideal $I = Rabc$ is weakly $1$-absorbing prime if and only if $abc = 0$.
\end{lem}
\begin{proof}
{Suppose that $R$ is a ring, $a,b,c \in J(R)$ and $I = Rabc$ is an ideal of $R$. One can easily see that $I$ is a weakly $1$-absorbing prime ideal of $R$, if $abc = 0$. So assume that $I$ is a weakly $1$-absorbing prime ideal of $R$ and $abc \neq 0$.  Hence $a,b,c$ are nonunit, and so either $ab \in I$ or $c \in I$. If $ab \in I$, then $ab = abcx$ for some $x\in R$, and thus $ab(1-cx) = 0$.
This follows that $ab=0$, since $cx \in J(R)$ which means that $1 - cx$ is  unit. Therefore $abc = 0$, a contradiction. If $c \in I$, by a similar argument, we can see that $c(1-abk) = 0$ for some $x\in R$. This implies that $c=0$, a contradiction. Thus $abc = 0$.
}
\end{proof}

\begin{thm}\label{442}
  Let $R$ be a  local ring with a unique maximal ideal $M$. Then every proper ideal of $R$ is
weakly $1$-absorbing prime if and only if $M^3 = \{0\}$.
\end{thm}
\begin{proof}
{Suppose that $R$ is local with a maximal ideal $M$ and every proper ideal of $R$ is weakly $1$-absorbing prime. It follows from Lemma \ref{44} that $abc = 0$ for every $a,b,c \in M$, since $abcR$ is a weakly $1$-absorbing prime ideal. Therefore $M^3 = \{0\}$. Conversely, suppose that $M^3 = \{0\}$ and $I$ is a nonzero proper ideal of $R$. Suppose also that $0\neq abc \in I$ for some nonunit elements $a,b,c$ of $R$. Then $a,b,c\in M$. But $M^3 = \{0\}$ and $0\neq abc$, hence either $a$ is a unit of $R$ or $b$ is a unit of $R$ or $c$ is a unit of $R$, a contradiction. Thus $I$ is a weakly $1$-absorbing prime ideal.
}
\end{proof}
\begin{cor} \label{3.3.3}
Suppose that $R$ is a local ring with a unique maximal ideal $M$ such that $M^2 = \{0\}$. Then
every proper ideal of $R$ is a $1$-absorbing prime ideal of $R$.
\end{cor}
\begin{proof}
{ Suppose that $R$ is local with a maximal ideal $M$ and $I$ is a proper ideal of $R$. It follows from Theorem \ref{442} that $I$ is weakly $1$-absorbing prime, because $M^3 = \{0\}$. Therefore assume that $0 = abc \in I$ for some nonunit elements $a,b,c$ of $R$. This means that $a,b,c \in M$. Since $M^2 = \{0\}$, $abc = 0$ and $M^2$ is primary, we conclude that either $c \in M^2\subseteq I$ or $ab \in M$. In the second case, it easy to see that $ab = 0\in I$. Thus $I$ is a
 $1$-absorbing prime ideal of $R$.
}
\end{proof}

\begin{thm} \label{yas4}
Suppose that $R_1$ and $R_2$ are  rings and $R = R_1\times R_2$. If $R_1$ and $R_2$ are local with
maximal ideals $M_1$ and $M_2$, respectively, and every proper
ideal of $R$ is a weakly $1$-absorbing prime ideal of $R$, then $M_1^2, M_2^2$ are zero ideals and either $R_1$ or $R_2$ is a field.
\end{thm}
\begin{proof}
{Suppose that $R_1$ and $R_2$ are local rings with maximal ideals $M_1$ and $M_2$, respectively, and $R = R_1\times R_2$. First, suppose that every proper ideal of $R$ is a weakly $1$-absorbing prime ideal of $R$ and $a,b \in M_1$ such that $ab \neq 0$. Then the ideal $I = abR_1 \times \{0\}$ of $R$ is a weakly $1$-absorbing prime ideal of $R$. Without loss of generality, we may assume that $a,b$ are nonunit. But $(0, 0) \neq (a, 1)(b, 1)(1, 0) = (ab, 0) \in I$ and $I$ is a weakly $1$-absorbing prime ideal of $R$, hence $(1, 0)\in I$, since $(a,1)(b,1)\notin I$, and so $1 = abx$ for some $x \in R_1$. This means that $1 - abx = 0$, a contradiction, as $1 - abx$ is a unit element of $R_1$. Therefore $M_1^2 = \{0\}$. By a similar argument as above and taking the ideal $I = \{0\} \times abR_1$, one can see that $M_2^2 = \{0\}$. Now, assume that $R_1$ and $R_2$ are not fields and look for a contradiction. It follows that $M_1 \neq \{0\}$, and so the ideal $J = M_1 \times \{0\}$ is a weakly $1$-absorbing prime ideal of $R$.	But $M_2^2 = \{0\}$ and $R_2$ is not a field, hence there exists $c \in M_2$ such that $c \neq 0$. Suppose that $0\neq m_1 \in M_1$. Then $(0, 0)\neq (m_1,1)(1,c)(1,c) = (m_1,c^2) = (m_1, 0) \in J = M_1 \times \{0\}$ which is a contradiction, since both $(m_1, 1)(1, c) = (m_1, c) \notin J$ and $(1,c)\notin J$. Thus either $R_1$ or $R_2$ is a field.
}
\end{proof}


Finally, we characterize rings in which every ideal is weakly $1$-absorbing prime. To this end, the following lemma is needed.

\begin{lem} \label{yas5}
$(1)$ Suppose that $R = R_1 \times R_2 \times R_3$, where $R_1$, $R_2$ and $R_3$ are  rings. If every proper ideal of $R$ is weakly $1$-absorbing prime, then $R_1, R_2, R_3$ are fields.

$(2)$ Let $R$ be a ring and every proper ideal of $R$ be weakly $1$-absorbing prime. Then $R$ has at most three maximal ideals.
\end{lem}
\begin{proof}
{$(1)$ Let $R = R_1 \times R_2 \times R_3$ be a decomposable ring, where $R_1$, $R_2$ and $R_3$ are rings such that every proper ideal of $R$ is weakly $1$-absorbing prime. Without loss of generality, assume that $R_1$ is not a field. Let $J$ be a non-zero proper ideal of $R_1$ and $m$ be a non-zero element of $J$. Suppose that $I = J \times \{0\} \times \{0\}$. Then $I$ is a weakly $1$-absorbing prime ideal of $R$. Since $(0,0,0) \neq (m, 1, 1)(1, 0, 1)(1, 1, 0) = (m, 0, 0) \in I$ and both $(m, 1, 1)(1, 0, 1) = (m, 0, 1) \notin I$ and $(1,1,0) \notin I$, we have a contradiction. Hence $R_1, R_2, R_3$ are fields.

$(2)$ Let $M_1, M_2, M_3, M_4$ be distinct maximal ideals of $R$ and  $I = M_1\cap M_2\cap M_3$. Since every weakly $1$-absorbing prime is weakly $2$-absorbing, and by \cite[Theorem 2.5]{Badawi}, $I$ is not a $2$-absorbing ideal of $R$, we conclude that $I$ is a weakly $2$-absorbing ideal of $R$ that is not a $2$-absorbing ideal of $R$. Therefore, by
\cite[Theorem 2.4]{Darani}, $I^3 = \{0\}$, a contradiction. Thus $R$ has at most
three distinct maximal ideals.
}
\end{proof}
We end this paper with the following result.
\begin{thm} \label{yas6}
Let $R$ be a ring. If $(1)$ every proper ideal of $R$ is weakly $1$-absorbing prime, then either $(2)$ $R$ is a local ring with a unique maximal ideal $M$ such that $M^3 = 0$ or $(3)$ $R = R_1\times R_2$, where $R_1$ is a local ring with a unique
maximal ideal $M_1$ such that $M_1^2 = 0$ and $R_2$ is a field or $(4)$ $R = R_1 \times R_2 \times R_3$, where $R_1, R_2, R_3$ are fields. Furthermore, $(2) \Rightarrow (1)$ but $(3) \Rightarrow (1)$ and $(4) \Rightarrow (1)$ need not necessarily be true.
\end{thm}

\begin{proof}
{ If $R$ satisfies condition $(2)$, then the condition $(1)$ follows from Theorem \ref{442}. Suppose $R$ satisfies condition $(3)$, i.e., $R = R_1\times R_2$, where $R_1$ is local with a unique maximal ideal $M_1$ such that $M_1^2 = 0$ and $R_2$ is a field. Since $R_2$ is a field and every proper ideal of $R_1$ is a $1$-absorbing prime, we conclude that the ideals $\{0\} \times R_2$ and $R_1 \times \{0\}$ are weakly $1$-absorbing prime ideals of $R$. Suppose that $J$ is a non-zero proper ideal of $R_1$. It follows from Theorem \ref{311} that $J \times R_2$ is a weakly $1$-absorbing prime ideal, since $J$ is a $1$-absorbing prime ideal of $R_1$, by Corollary \ref{3.3.3}. Hence, it is enough to prove that $J \times \{0\}$ is a weakly $1$-absorbing prime ideal of $R$. Let $0\neq a\in J$. Then $(0, 0, 0)\neq (a, 1)(a, 1)(1, 0) = (a^2, 0, 0)\in J \times \{0\}$ and both $(a^2, 1) \notin J \times \{0\}$ and $(1, 0) \notin J \times \{0\}$. Thus the condition $(3) \Rightarrow (1)$ need not necessarily be true. Suppose that $R$ satisfies condition $(4)$, i.e., $R = R_1 \times R_2 \times R_3$, where $R_1, R_2, R_3$ are fields. By taking $I = R_1 \times \{0\} \times \{0\}$, where $R_1$ is a field, one can easily see that the ideal $I$ is not weakly $1$-absorbing prime, because $(0, 0, 0)\neq (1, 0, 1)(1, 0, 1)(1, 1, 0) = (1, 0, 0)\in I$ and neither $(1, 0, 1)(1, 0, 1) = (1, 0, 1)\in I$ nor $(1, 1, 0)\in I$, showing that the condition $(4) \Rightarrow (1)$ need not necessarily be true. To complete the proof, assume that every proper ideal of $R$ is weakly $1$-absorbing prime. It follows from Lemma \ref{yas5} part $(2)$ that $R$ has at most three maximal ideals, and so if $R$ is local with a unique maximal $M$, then by Theorem \ref{442}, $M^3 = \{0\}$. Thus $R$ satisfies $(2)$. If $R$ has two maximal ideals $M_1$ and $M_2$, then $J(R) = M_1\cap M_2$ is a weakly $1$-absorbing prime ideal of $R$. Since every proper ideal of $R$ is weakly $1$-absorbing prime, $Rabc$ is  weakly $1$-absorbing prime, for every $a,b,c\in J(R)$. Hence $abc = 0$, by Lemma \ref{44}, and this means that $J(R)^3 = M_1^3\cap M_2^3 = \{0\}$. Thus, by the Chinese Remainder Theorem $R \cong R/M_1^3 \times R/M_2^3$. But $R/M_1^3$ and $R/M_2^3$ are local rings and every proper ideal of $R$ is  weakly $1$-absorbing prime. Hence, by Theorem \ref{yas4},  either $R/M_1^3$ or $R/M_2^3$ is a field and $N^2 = J^2 = 0$, where $N$ and $J$ are the maximal ideals of $R/M_1^3$ and $R/M_2^3$, respectively. Thus $R$ satisfies $(3)$. If $R$ has three maximal ideals $M_1, M_2$ and $M_3$, then
$J(R) = M_1\cap M_2\cap M_3$ is a weakly $1$-absorbing prime ideal of $R$ (weakly $2$-absorbing ideal). But $J(R)$ is not a $2$-absorbing, hence, by \cite[Theorem 2.4]{Darani}, $J(R)^3 = M_1^3\cap M_2^3 \cap M_3^3 = \{0\}$. Thus, by the Chinese Remainder Theorem, $R \cong R/M_1^3 \times R/M_2^3 \times R/M_3^3$. Since every proper ideal of $R$ is weakly $1$-absorbing prime, Lemma \ref{yas5} implies that $R/M_1^3$, $R/M_2^3$ and $R/M_3^3$ are fields. Thus $R$ satisfies $(4)$.
}
\end{proof}

\end{document}